\documentclass[12pt]{amsart}
\usepackage{times}
\usepackage[T1]{fontenc}
\usepackage{dsfont}
\usepackage{mathrsfs}
\usepackage[colorlinks]{hyperref}
\usepackage{xcolor}
\usepackage[a4paper,asymmetric]{geometry}
\usepackage{mathscinet}
\usepackage{latexsym}
\usepackage{amsthm}
\usepackage{amssymb}
\usepackage{amsfonts}
\usepackage{amsmath}
\newtheorem{theorem}{Theorem}[section]
\newtheorem{thm}[theorem]{Theorem}
\newtheorem{lemma}[theorem]{Lemma}
\newtheorem{lem}[theorem]{Lemma}
\newtheorem{proposition}[theorem]{Proposition}

\newtheorem{corollary}[theorem]{Corollary}

\theoremstyle{definition}

\newtheorem{defn}[theorem]{Definition}

\theoremstyle{remark}

\newtheorem{rem}[theorem]{Remark}
\numberwithin{equation}{section}

 \DeclareMathAlphabet{\mathpzc}{OT1}{pzc}{m}{it}

  \newcommand{\dif}{\mathrm{d}}

 \newcommand{\E}{\mathbb{E}}            

 \newcommand{\R}{\mathbb{R}}
 
 \newcommand{\PP}{\mathbb{P}}
 \newcommand{\mcl}{\mathcal}
 
 \newcommand{\Be}{\begin{equation}}
 \newcommand{\Ee}{\end{equation}}
 \newcommand{\Bs}{\begin{split}}
 \newcommand{\Es}{\end{split}}
  \newcommand{\Bes}{\begin{equation*}}
 \newcommand{\Ees}{\end{equation*}}
 \newcommand{\BT}{\begin{thm}}
 \newcommand{\ET}{\end{thm}}
 \newcommand{\Bp}{\begin{proof}}
 \newcommand{\Ep}{\end{proof}}
 \newcommand{\BL}{\begin{lem}}
 \newcommand{\EL}{\end{lem}}
 \newcommand{\BP}{\begin{proposition}}
 \newcommand{\EP}{\end{proposition}}
 \newcommand{\BC}{\begin{corollary}}
 \newcommand{\EC}{\end{corollary}}
 \newcommand{\BR}{\begin{rem}}
 \newcommand{\ER}{\end{rem}}
 \newcommand{\BD}{\begin{defn}}
 \newcommand{\ED}{\end{defn}}
 \newcommand{\BI}{\begin{itemize}}
 \newcommand{\EI}{\end{itemize}}
 
 \newcommand{\tl}{\tilde}

\begin{document}
\title[Smooth densities of perturbed diffusion processes]
{Smooth densities of the laws of perturbed diffusion processes}
\author[L. Xu]{Lihu Xu}
\address{Department of Mathematics,
Faculty of Science and Technology,
University of Macau, E11
Avenida da Universidade, Taipa,
Macau, China}
\email{lihuxu@umac.mo}
\author[W. Yue]{Wen Yue}
\address{Institute of Analysis and Scientific Computing,
Vienna University of Technology,
Wiedner Hauptstrass 8-10,
1040 Wien, Austria
}
\email{wenyue@hotmail.co.uk}

\author[T.Zhang]{Tusheng Zhang}
\address{School of Mathematics,University of Manchester
Oxford Road,
Manchester M13 9PL,
United Kingdom}
\email{tusheng.zhang@manchester.ac.uk}
\thanks{Lihu Xu is supported by the following grants: Science and Technology Development Fund Macao S.A.R FDCT  049/2014/A1,  MYRG2015-00021-FST}


\begin{abstract} \label{abstract}
Under some regularity conditions on $b$, $\sigma$ and $\alpha$, we prove that the following perturbed stochastic differential equation
\Be \label{e:PerEqn}
X_t=x+\int_0^t b(X_s) \dif s+\int_0^t \sigma(X_s) \dif B_s+\alpha \sup_{0 \le s \le t} X_s, \ \ \ \alpha<1
\Ee
admits smooth densities for all $0 \le t \le t_0$, where $t_0>0$ is some finite number.  
\end{abstract}
\maketitle
\section{Introduction}
There have been a considerable body of literatures devoted to the study of perturbed stochastic differential equations(SDEs), see \cite{CPY}-\cite{GY1},\cite{PW},\cite{YZ}. Let $(\Omega, \mathcal{F}, \{\mathcal{F}_{t}\}_{t\geq 0},\PP)$ be a filtered probability space with filtration$\{\mathcal{F}_{t}\}_{t\geq 0}$, let $\{B_{t}\}_{t\geq0}$ be a one-dimensional standard $\{\mathcal{F}_{t}\}_{t\geq 0}$-Brownian Motion. Suppose that $\sigma(x), b(x)$ are Lipschitz continuous functions on $\R$. It was proved in \cite{RDZ} that the following perturbed stochastic differential equation:
\Be \label{e:PerEqn}
X_t=x+\int_0^t b(X_s) \dif s+\int_0^t \sigma(X_s) \dif B_s+\alpha \sup_{0 \le s \le t} X_s,\ \   \forall \ \alpha<1,
\Ee
admits a unique solution. If $|\sigma(x)|>0$, it was shown in \cite{YZ} that the law of $X_t$ is absolutely continuous with respect to Lebesgue measure, i.e. the law of $X_t$ admits a density for $t>0$.
\vskip 3mm

There seem no smooth density results for the law of a perturbed diffusion process, this paper aims to partly fill in this gap. The smoothness of  densities is a popular topic in stochastic analysis and has been intensively studied for several decades, we refer readers to \cite{N}, \cite{S} and references therein.
Our approach to proving the smoothness of densities is by Malliavin calculus, so let us first recall some well known results on Malliavin calculus \cite{N} to be used in this paper.
\vskip 3mm

Let $\Omega=C_{0}(\R_{+})$ be the space of continuous functions on $\R_{+}$ which are zero at zero. Denote by $\mathcal{F}$ the Borel $\sigma$-field on $\Omega$ and $\PP$ the Wiener measure, then the canonical coordinate process \{$\omega_{t}$, $t\in \R_{+}$\} on $\Omega$ is a Brownian motion $B_t$.
 Define $\mathcal{F}_{t}^{0}=\sigma(B_{s},s\leq t)$ and denote by $\mathcal{F}_{t}$ the completion of $\mathcal{F}_{t}^{0}$ with respect to the $\PP$-null sets of $\mathcal{F}.$

Let $H:=L^{2}(\R_{+},\mathcal{B},\mu)$ where $(\R_{+},\mathcal{B})$ is a measurable space with $\mathcal{B}$ being the Borel $\sigma$-field of $\R_{+}$ and $\mu$ being the Lebesgue measure on $\R_{+}$, we denote the norm of $H$ by $\|.\|_H$. For any $h \in H$, $W(h)$ is defined by
\begin{equation}
W(h)=\int_{0}^{\infty}h(t)\dif B_{t},
\end{equation}
note that $\{W(h),h\in H\}$ is a Gaussian Process on $H$.\\

We denote by $C_{p}^{\infty}(\mathbb{R}^{n})$ the set of all infinitely differentiable functions $f:\mathbb{R}^{n}\rightarrow\mathbb{R}$ such that $f$ and all of its partial derivatives have polynomial growth.
Let $\mcl S$ be the set of smooth random variables defined by
\begin{equation*}
\mcl S=\{F=f(W(h_{1}),...,W(h_{n}));\ h_{1},...,h_{n}\in H,n\geq 1, f\in C_{p}^{\infty}(\R^{n})\}.
\end{equation*}
Let $F\in \mcl S$, define its Malliavin derivative $D_{t}F$ by
\begin{equation}
D_{t}F=\sum_{i=1}^{n}\partial_{i}f(W(h_{1}),...,W(h_{n}))h_{i}(t),
\end{equation}
and its norm by
\
\begin{equation*}
||F||_{1,2}=[\E(|F|^{2})+\E(||DF||_H^{2})]^{\frac{1}{2}},
\end{equation*}
where $||DF||_H^{2}=\int_0^\infty |D_t F|^2 \mu(\dif t)$.
Denote by $\mathbb{D}^{1,2}$ the completion of $\mcl S$ under the norm $\|.\|_{1,2}$. We further
define the norm
\
\begin{equation*}
||F||_{m,2}=\left[\E(|F|^{2})+\sum_{k=1}^m \E(\|D^kF\|_{H^{\otimes k}}^{2})\right]^{\frac{1}{2}}.
\end{equation*}
Similarly, $\mathbb{D}^{m,2}$ denotes the completion of $\mcl S$ under the norm $||.||_{m,2}$.
\vskip 3mm

 We shall use the following two propositions:

\begin{proposition}[Proposition 1.2.3 of \cite{N}]\label{estimate of composition}
Let $\phi:\R^d\to R$ be a continuously differentiable function with bounded partial derivatives. Suppose that $F=(F^1,\cdots,F^d)$ is a random vector whose components belong to the space $\mathbb{D}^{1,2}$. Then $\phi(F) \in \mathbb{D}^{1,2}$, and
\begin{equation*}
D(\phi(F))=\sum_{i=1}^{d}\partial_i\phi(F)DF^i.
\end{equation*}
\end{proposition}

\begin{proposition}[Proposition 2.1.5 of \cite{N}] \label{t:Cri}
If $F \in \mathbb{D}^{\infty,2}$ with $\mathbb{D}^{\infty,2}=\cap_{m \ge 1} \mathbb{D}^{m,2}$
and $\|DF\|_H^{-1} \in \cap_{p\geq 1}L^p(\Omega)$, then the density of $F$ belongs to the infinitely continuously differentiable function space $C^\infty(\R)$.
\end{proposition}
Throughout this paper, for a bounded measurable function $f$, we shall denote
$$\|f\|_\infty=\sup_{x \in \R} |f(x)|. $$

\section{Main Results}
Throughout this paper, we need to assume $\alpha<1$ to guarantee that Eq. \eqref{e:PerEqn} has a unique solution \cite{RDZ}. Furthermore, it is shown in \cite{YZ} that
\
\begin{thm}(\cite[Theorem 3.1]{YZ})
Let $(X_t)_{t \ge 0}$ be the unique solution to Eq. \eqref{e:PerEqn}. Then $X_t \in \mathbb{D}^{1,2}$ for all $t>0$.
\end{thm}
\begin{thm} (\cite[Theorem 3.2]{YZ})
Assume that $\sigma$ and $b$ are both Lipschitz continuous, and $|\sigma(x)|>0$
for all $x \in \R$. Then, for $t>0$, the law of $X_t$ is absolutely continuous with respect to
Lebesgue measure.
\end{thm}

In this paper, we shall prove the following results about the smoothness of densities:
\begin{thm} \label{t:MThm1}
Assume that $b$ is bounded smooth and that $\sigma(x)\equiv \sigma$. If $\alpha<1$, $t_0>0$ and $b$ satisfy
$$\theta(t_0,\alpha,b)<1/2,$$
with $\theta(t_0,\alpha,b):=\sqrt{2\|b'\|_{{\rm \infty}}^2 t_0^2+8 \alpha^2}+{\|b'\|_{{\rm \infty}}^2 t_0^2+4\alpha^2}$, then the law of $X_t$ in \eqref{e:PerEqn} admits a smooth density for all $t \in (0,t_0]$.
\end{thm}

\begin{thm} \label{t:MThm2}
Assume that $b$ is bounded smooth, and $\sigma$ is bounded smooth with $\|\sigma'\|_\infty<\infty$, $\|\sigma''\|_\infty<\infty$ and $\inf_{x\in \R}|\sigma(x)|>0$.
Let
\begin{equation}
F(y)=\int_x^{y} \frac{1}{\sigma(u)} \dif u, \ \ \ \ y \in (-\infty,\infty)
\end{equation}
and $\tl b(x)=\frac{b(F^{-1}(x))}{\sigma(F^{-1}(x))}-\frac 12 \sigma'(F^{-1}(x))$, then $\tl b$ is bounded smooth with $\|\tl b'\|_\infty<\infty$.
If $\alpha<1$, $t_0>0$ and $b$ satisfy
$$\theta(t_0,\alpha,\tl b)<1/2$$
with $\theta(t_0,\alpha,\tl b):=\sqrt{2\|\tl b'\|_{{\rm \infty}}^2 t_0^2+8 \alpha^2}+{\|\tl b'\|_{{\rm \infty}}^2 t_0^2+4\alpha^2}$,
 then the law of $X_t$ in \eqref{e:PerEqn} admits a smooth density for all $t \in (0,t_0]$.
\end{thm}

\begin{proof} [{\bf Proofs of Theorems \ref{t:MThm1} and \ref{t:MThm2}:}]

The main idea is to use Proposition \ref{t:Cri} to prove the two theorems.
To verify the conditions in Proposition \ref{t:Cri},
it suffices to prove that $X_t \in D^{m,2}$ for all $m \ge 1$ and $\|DX_t\|_H \ge c>0$ a.s. for some constant $c>0$.

Theorem \ref{t:MThm1} immediately follows from Lemmas \ref{l:XtDm2} and \ref{t:SthDen} below.

Now we prove Theorem \ref{t:MThm2}. Recall $Y_t=\int_x^{X_t} \frac{1}{\sigma(u)} du$ in Lemma \ref{t:SthDenY} below, by the condition of $\sigma$, $F$ is a continuous and strictly increasing function with bounded derivative and thus
\Be
\|D Y_t\|_H=\|D F(X_t)\|_H \le \frac{1}{\inf_{x\in \R}|\sigma(x)|} \|D X_t\|_H.
\Ee
Hence, by Lemmas \ref{l:XtDm2} and \ref{t:SthDenY} below, under the same condition as in Theorem \ref{t:MThm2} we have
\Be
\|D X_t\|_H \ge \inf_{x\in \R}|\sigma(x)|\cdot \|D Y_t\|_H \ge  \inf_{x\in \R}|\sigma(x)| \cdot \frac{[1-2 \theta(t_0,\alpha,\tl b)]  t}{2 (1+2 \|\tl b'\|_{{\rm \infty}}^2 t^2+2 \alpha^2)}  \ \ \ \ \ t \in [0,t_0].
\Ee
Hence, $X_t$ admits a smooth density for all $t \in (0,t_0]$.
\end{proof}

\section{Auxiliary lemmas}
It is well known that $\|DX_t\|_H$ has the following representation \cite{YZ} for all $t>0$:
$$\|DX_t\|_H=\left(\int_0^t |D_r X_t|^2 \dif r\right)^{\frac 12}$$
with $D_rX_t$ satisfying
\Be \label{e:DrXt}
D_r X_t=\sigma(X_r)+\int_r^t D_r b(X_s) \dif s+\int_r^t D_r \sigma(X_s)\dif B_s+\alpha D_r \left(\sup_{0 \le s \le t} X_s\right).
\Ee
We shall often use the following fact (\cite{YZ}, \cite{N})
\Be \label{e:DrXt=0}
D_r X_t=0 \ \ \ \ {\rm if}\ \ r>t,
\Ee
\Be \label{e:MalSupLes}
\left\|D(\sup_{0 \le s \le t} X_s)\right\|_H \le \sup_{0 \le s \le t} \left\|D X_s\right\|_H,
\Ee
where
\begin{equation*}
\left\|D(\sup_{0 \le s \le t} X_s)\right\|_H^2=\int_0^t \left|D_r \left(\sup_{0 \le s \le t} X_s\right)\right|^2 \dif r,
\ \ \ \ \left\|D X_t\right\|_H^2=\int_0^t |D_r X_t|^2 \dif r.
\end{equation*}
\ \ \  \\

\subsection{$X_t$  is an element in $\mathbb{D}^{m,2}$ for all $t>0$ and $m \ge 1$}

\begin{lemma} \label{l:XtDm2}
Let $X_t$ be the solution of the perturbed stochastic differential equation \eqref{e:PerEqn}, and suppose that the coefficients $b$ and $\sigma$ are smooth with bounded derivatives of all orders. Then $X_t$ belongs to $\mathbb{D}^{m,2}$ for all $t>0$ and all $m \ge 1$.
\end{lemma}

\begin{proof}
We shall use Picard iteration to prove the lemma. Letting $X^0_t=x_0$ for all $t>0$,
 define $X_t^{n+1}$ be the unique, adapted solution to the following equation:
 \
\begin{equation} \label{e:Pic1}
X_t^{n+1}=x_0+\int_{0}^{t}\sigma(X_s^n)\dif B_s+\int_{0}^{t}b(X_s^n) \dif s+\alpha \max_{0\leq s\leq t} \left(X_s^{n+1}\right),
\end{equation}
which obviously implies
\
\begin{equation*}
\max_{0 \le s \le t}\left(X_s^{n+1}\right)=x_0+\max_{0 \le s \le t}\left(\int_{0}^{t}\sigma(X_s^n)\dif B_s+\int_{0}^{t}b(X_s^n) \dif s\right)+\alpha \max_{0\leq s\leq t} \left(X_s^{n+1}\right).
\end{equation*}
Therefore,
\
\begin{equation*}
\max_{0 \le s \le t}\left(X_s^{n+1}\right)=\frac{x_0}{1-\alpha}+\frac1{1-\alpha}\max_{0 \le s \le t}\left(\int_{0}^{t}\sigma(X_s^n)\dif B_s+\int_{0}^{t}b(X_s^n)\dif s\right),
\end{equation*}
this and \eqref{e:Pic1} further gives
\begin{equation*}
\begin{split}
X_t^{n+1}=&\frac{x_0}{1-\alpha}+\int_{0}^{t}\sigma(X_s^n)\dif B_s+\int_{0}^{t}b(X_s^n)\dif s\\
&+\frac{\alpha}{1-\alpha}\max_{0\leq s\leq t}\left(\int_{0}^{s}\sigma(X_u^n)\dif B_u+\int_{0}^{s}b(X_u^n)\dif u\right).
\end{split}
\end{equation*}
By the above representation of $X^{n+1}_t$ and a standard method \cite{RDZ}, for every $t>0$ we have
\
\Be  \label{e:XnConXt}
\lim_{n \rightarrow \infty} X^n_t=X_t \ \ \ \ {\rm in} \ L^2(\Omega).
\Ee
Let $m \ge 1$, it is standard to check that $X^n_t \in \mathbb{D}^{m,2}$ for every $t>0$ and $n \ge 1$ \cite[Theorem 3.1]{YZ}.
By a similar argument as in \cite[Theorem 3.1]{YZ},
we have
\
\Be  \label{e:Pic2}
\sup_{n \ge 1}\E\left[\|D^k X^n_t\|_{H^{\otimes k}}^2\right] <\infty, \ \ \ \ \ k=1,...,m.
\Ee

 Next we prove $X_t \in \mathbb{D}^{m,2}$ by the argument of \cite[Proposition 1.2.3]{N}. Indeed, by \eqref{e:Pic2}, there exists some subsequence $D^k X^{n_j}_t$ weakly converges to some
$\alpha_k$ in $L^2(\Omega,H^{\otimes k})$ for $k=1,...,m$. By \eqref{e:XnConXt} and the remark immediately below \cite[Proposition 1.2.2]{N}, the projections of $D^k X^{n_j}_t$ on any Wiener chaos converge in the weak topology of $L^2(\Omega)$, as $n_j$ tends to infinity, to those of
$\alpha_k$ for $k=1,...,m$. Hence, $X_t \in \mathbb{D}^{m,2}$ and $D^k X_t=\alpha_k$ for $k=1,...,m$.  Moreover, for any weakly convergent subsequence the limit must be equal to $\alpha_1,..., \alpha_m$ by the same argument as above, and this implies the weak convergence of the whole sequence.
\end{proof}
\ \

\subsection{Additive noise case}
If $\sigma(x) \equiv \sigma$, then Eq. \eqref{e:DrXt} reads as
\Be \label{e:MalDer}
D_r X_t=\sigma+\int_r^t D_r b(X_s) \dif s+ \alpha D_r \left(\sup_{0 \le s \le t} X_s\right).
\Ee

\begin{lem} \label{l:DrXDifEst}
Let $t>0$ be arbitrary and $b$ be bounded smooth with $\|b'\|_\infty<\infty$. For all $0<t_1<t_2 \le t$, we have
\Bes
\begin{split}
\left|\|D X_{t_2}\|_H^2-\|D X_{t_1}\|_H^2\right| \le 2\left[\sqrt{2\|b'\|_{{\rm  \infty}}^2(t_2-t_1)^2+8\alpha^2}+{\|b'\|_{{\rm \infty}}^2(t_2-t_1)^2+4\alpha^2}\right]\sup_{0 \le s \le t} \left\|D X_s\right\|_H^2.
\end{split}
\Ees
\end{lem}

\begin{proof}
It is easy to see that
\Bes
\left|\|D X_{t_2}\|_H^2-\|D X_{t_1}\|_H^2\right|=\left|\int_0^{t_2} (D_r X_{t_2})^2 \dif r-\int_0^{t_1} (D_r X_{t_1})^2 \dif r\right| \le I_1+I_2,
\Ees
where
$$I_1:=\int_{t_1}^{t_2} (D_r X_{t_2})^2 \dif r, \ \ \ \ \ I_2:=\int_{0}^{t_1} \left|(D_r X_{t_2})^2-(D_r X_{t_1})^2\right| \dif r.$$
We claim that
\Be \label{e:DrXDifEst}
\begin{split}
\int_0^{t_2} (D_r X_{t_2}-D_r X_{t_1})^2 \dif r & \le 2 \left[\|b'\|_{{\rm \infty}}^2(t_2-t_1)^2+4\alpha^2\right] \sup_{0 \le s \le t} \left\|D X_s\right\|_H^2.
\end{split}
\Ee
and we will prove it in the last part of this proof.

Let us now estimate $I_1$ and $I_2$ by \eqref{e:DrXDifEst}. Observe
$$I_1=\int_{t_1}^{t_2} (D_r X_{t_2}-D_r X_{t_1})^2 \dif r \le \int_{0}^{t_2} (D_r X_{t_2}-D_r X_{t_1})^2 \dif r,$$
by \eqref{e:DrXDifEst} we have
\Be
\begin{split}
I_1 \le 2 \left[\|b'\|_{{\rm \infty}}^2(t_2-t_1)^2+4\alpha^2\right] \sup_{0 \le s \le t} \left\|D X_s\right\|_H^2.
\end{split}
\Ee
Further observe
\Bes
\begin{split}
I_2 & \le \left[\int_{0}^{t_1} (D_r X_{t_2}-D_r X_{t_1})^2 \dif r\right]^{\frac 12}\left[\int_{0}^{t_1} |D_r X_{t_2}+D_r X_{t_1}|^2 \dif r\right]^{\frac 12} \\
& \le \sqrt 2 \left[\int_{0}^{t_1} (D_r X_{t_2}-D_r X_{t_1})^2 \dif r\right]^{\frac 12}\left[\int_{0}^{t_1} |D_r X_{t_2}|^2+|D_r X_{t_1}|^2 \dif r\right]^{\frac 12} \\
& \le \sqrt 2 \left[\int_{0}^{t_1} (D_r X_{t_2}-D_r X_{t_1})^2 \dif r\right]^{\frac 12}\left[\int_{0}^{t_2} |D_r X_{t_2}|^2 \dif r+\int_0^{t_1} |D_r X_{t_1}|^2 \dif r\right]^{\frac 12}\\
& \le 2 \left[\int_{0}^{t_1} (D_r X_{t_2}-D_r X_{t_1})^2 \dif r\right]^{\frac 12} \sup_{0 \le s \le t} \left\|D X_s\right\|_H \\
& \le 2 \left[\int_{0}^{t_2} (D_r X_{t_2}-D_r X_{t_1})^2 \dif r\right]^{\frac 12} \sup_{0 \le s \le t} \left\|D X_s\right\|_H,
\end{split}
\Ees
this inequality and \eqref{e:DrXDifEst} gives
\Bes
\begin{split}
I_2 \le 2\sqrt{2[\|b'\|_{{\rm \infty}}^2(t_2-t_1)^2+4\alpha^2]} \sup_{0 \le s \le t} \left\|D X_s\right\|_H^2.
\end{split}
\Ees
Combining the estimates of $I_1$ and $I_2$, we immediately get the desired inequality in the lemma.

It remains to prove \eqref{e:DrXDifEst}. By \eqref{e:MalDer}, we have
\Bes
\begin{split}
(D_r X_{t_2}-D_r X_{t_1})^2 & \le 2 \left|\int_{t_1}^{t_2} D_r b(X_s) \dif s\right|^2+2 \alpha^2 \left|D_r \left(\sup_{0 \le s \le t_1} X_s\right)-D_r \left(\sup_{0 \le s \le t_2} X_s\right)\right|^2 \\
& \le 2 \left|\int_{t_1}^{t_2} D_r b(X_s) \dif s\right|^2+4\alpha^2 \left|D_r \left(\sup_{0 \le s \le t_1} X_s\right)\right|^2+4 \alpha^2 \left|D_r \left(\sup_{0 \le s \le t_2} X_s\right)\right|^2.
\end{split}
\Ees
By H\"{o}lder inequality, \eqref{e:DrXt=0} and Proposition \ref{estimate of composition}, we have
\Bes
\begin{split}
\int_0^{t_2} \left|\int_{t_1}^{t_2} D_r b(X_s) \dif s\right|^2 \dif r  & \le  \|b'\|_{{\rm \infty}}^2 \int_0^{t_2} (t_2-t_1) \int_{t_1}^{t_2} |D_r X_s|^2 \dif s \dif r \\
& = \|b'\|_{{\rm \infty}}^2 (t_2-t_1) \int_{t_1}^{t_2} \int_0^s |D_r X_s|^2 \dif r \dif s \\
& \le  \|b'\|_{{\rm \infty}}^2(t_2-t_1)^2 \sup_{0 \le s \le t} \left\|D X_s\right\|_H^2.
\end{split}
\Ees
Moreover, by \eqref{e:MalSupLes} and \eqref{e:DrXt=0} we have
\Bes
\int_0^{t_2} \left|D_r \left(\sup_{0 \le s \le t_2} X_s\right)\right|^2 \dif r \le \sup_{0 \le s \le t_2} \|D X_s\|_H^2 \le \sup_{0 \le s \le t} \|D X_s\|_H^2,
\Ees

\Bes
\int_0^{t_2} \left|D_r \left(\sup_{0 \le s \le t_1} X_s\right)\right|^2 \dif r = \int_0^{t_1} \left|D_r \left(\sup_{0 \le s \le t_1} X_s\right)\right|^2 \dif r \le \sup_{0 \le s \le t} \|D X_s\|_H^2.
\Ees
Collecting the above four inequalities, we immediately get the desired \eqref{e:DrXDifEst}.
\end{proof}

\begin{lem}  \label{l:LowUppBouDrX}
Let $b$ be bounded smooth with $\|b'\|_\infty<\infty$, we have
\Be  \label{e:LowUppBouDrX}
\sup_{0 \le s \le t} \|DX_s\|_H^2 \ge \frac{\sigma^2 t}{2 (1+2\|b'\|_{{\rm \infty}}^2 t^2+2 \alpha^2)}, \ \ \ \ \ t>0.
\Ee
\end{lem}

\begin{proof}
By \eqref{e:MalDer} and using $(a+b)^2\geq \frac{1}{2}a^2-b^2$, we have
\Bes
\begin{split}
(D_r X_t)^2 & \ge \frac 12 \sigma^2-\left[\int_r^t D_r b(X_s) \dif s+\alpha D_r \left(\sup_{0 \le s \le t} X_s\right)\right]^2 \\
& \ge \frac 12 \sigma^2-2 \left(\int_r^t D_r b(X_s) \dif s\right)^2-2\alpha^2 \left[D_r \left(\sup_{0 \le s \le t} X_s\right)\right]^2.
\end{split}
\Ees
Further observe
\begin{equation}   \label{e:MalDerDb}
\begin{split}
\int_0^t\left(\int_r^t D_r b(X_s) \dif s\right)^2 \dif r
& \le  \int_0^t (t-r) \int_r^t |D_r b(X_s)|^2 \dif s  \dif r \\
& \le   \int_0^t (t-r) \|b'\|_{{\rm \infty}}^2 \int_r^t |D_r X_s|^2 \dif s \dif r \\
& \le   t \|b'\|_{{\rm \infty}}^2 \int_0^t \int_r^t |D_r X_s|^2 \dif s \dif r \\
&=  t \|b'\|_{{\rm \infty}}^2\int_0^t \|DX_s\|_H^2  \dif s \\
&\le t^2 \|b'\|_{{\rm \infty}}^2  \sup_{0 \le s \le t}\|DX_s\|_H^2,
\end{split}
\end{equation}
where the second inequality is by Proposition \ref{estimate of composition}.
Hence,
\Bes
\begin{split}
\|D X_t\|_H^2
& \ge \frac{\sigma^2 t}{2}-2 \|b'\|_{{\rm \infty}}^2 t^2 \sup_{0 \le s \le t} \|DX_s\|_H^2-2\alpha^2 \|D (\sup_{0 \le s \le t}X_s) \|_H^2 \\
& \ge \frac{\sigma^2 t}{2}-2 \|b'\|_{{\rm \infty}}^2 t^2 \sup_{0 \le s \le t} \|DX_s\|_H^2-2\alpha^2 \sup_{0 \le s \le t}\|D X_s\|_H^2,
\end{split}
\Ees
where the last inequality is by \eqref{e:MalSupLes}.

This clearly implies
\Bes
\sup_{0 \le s \le t} \|DX_s\|_H^2 \ge \frac{\sigma^2 t}{2}-2\|b'\|_{{\rm \infty}}^2 t^2 \sup_{0 \le s \le t} \|DX_s\|_H^2-2\alpha^2 \sup_{0 \le s \le t} \|DX_s\|_H^2,
\Ees
which immediately yields the desired bound.
\end{proof}

\begin{lem} \label{t:SthDen}
Let $b$ is bounded smooth with $\|b'\|_\infty<\infty$ and $\sigma(x)\equiv \sigma$ with $\sigma \neq 0$.
If $\alpha<1$, $t_0>0$ and $b$ satisfy
$$\theta(t_0,\alpha,b)<1/2$$
with $\theta(r,\alpha,b):=\sqrt{2\|b'\|_{{\rm \infty}}^2 r^2+8 \alpha^2}+{\|b'\|_{{\rm \infty}}^2 r^2+4\alpha^2}$,
then
\begin{equation}
\|D X_{t}\|_H^2 \ge  \frac{[1-2 \theta(t_0,\alpha,b)] \sigma^2 t}{2 (1+2\|b'\|_{{\rm \infty}}^2 t^2+2 \alpha^2)}, \ \ \ \ \ t \in [0,t_0].
\end{equation}
\end{lem}

\begin{proof} Let $t \in [0,t_0]$.
For all $0 \le t_1 \le t_2 \le t$, by Lemma \ref{l:DrXDifEst}, we have
\Bes
\left|\|D X_{t_2}\|_H^2-\|D X_{t_1}\|_H^2\right| \le 2\theta(t_2-t_1,\alpha,b)\sup_{0 \le s \le t} \left\|D X_s\right\|_H^2.
\Ees
Hence, for all $s \in [0,t]$,
\Bes
\begin{split}
\|D X_{s}\|_H^2 & \le \left|\|D X_{s}\|_H^2-\|D X_{t}\|_H^2\right|+\|D X_{t}\|_H^2  \\
& \le 2 \theta(t-s,\alpha,b)\sup_{0 \le s \le t} \left\|D X_s\right\|_H^2+\|D X_{t}\|_H^2,
\end{split}
\Ees
and consequently
\Bes
\begin{split}
\sup_{0 \le s \le t} \|D X_{s}\|_H^2 \le 2\theta(t-s,\alpha,b)\sup_{0 \le s \le t} \left\|D X_s\right\|_H^2+\|D X_{t}\|_H^2.
\end{split}
\Ees
The above inequality and \eqref{e:LowUppBouDrX} further give
\
\Bes
\begin{split}
\|D X_{t}\|_H^2 & \ge \left[1-2\theta(t-s,\alpha,b)\right]\sup_{0 \le s \le t} \|D X_{s}\|_H^2 \\
& \ge \left[1-2\theta(t_0,\alpha,b)\right]\sup_{0 \le s \le t} \|D X_{s}\|_H^2.
\end{split}
\Ees
Combining the above inequality and Lemma \ref{l:LowUppBouDrX} immediately gives the desired inequality.
\end{proof}
\ \
\subsection{Multiplicative noise case}
 By the condition of $\sigma$, we have $\sup_{x \in \R} \sigma(x)<0$ or $\inf_{x \in \R} \sigma(x)>0.$
Without loss of generality, we assume that
 $$\inf_{x \in \R} \sigma(x)>0.$$
Let us consider the following well known transform
\begin{equation}
F(X_t)=\int_x^{X_t} \frac{1}{\sigma(u)} \dif u,
\end{equation}
it is easy to see that $F$ is a strictly increasing function with bounded derivative. Hence,
\begin{equation} \label{e:SupF}
\sup_{0 \le s \le t} F(X_s)=F \left(\sup_{0 \le s \le t} X_s\right).
\end{equation}
By It\^{o} formula, we have
\begin{equation}
 F(X_t)=\int_0^t \left(\frac{b(X_s)}{\sigma(X_s)}-\frac 12 \sigma'(X_s)\right) \dif s+B_t+\alpha\int_0^t \frac{1}{\sigma(X_s)} \dif M_s
\end{equation}
where $M_t=\sup_{0 \le s \le t} X_s$. It is easy to see that $M_t$ is an increasing function of
$t$ and that $\frac{1}{\sigma(X_s)}$ has a contribution to the related integral only when $X_s=M_s$.
Hence,
\begin{equation}
 F(X_t)=\int_0^t \left(\frac{b(X_s)}{\sigma(X_s)}-\frac 12 \sigma'(X_s)\right) \dif s+B_t+\alpha\int_0^t \frac{1}{\sigma(M_s)} \dif M_s.
\end{equation}
Since $M_t$ is a continuous increasing function with respect to $t$, we have
\begin{equation}
 F(X_t)=\int_0^t \left(\frac{b(X_s)}{\sigma(X_s)}-\frac 12 \sigma'(X_s)\right) \dif s+B_t+\alpha \int_0^{M_t} \frac{1}{\sigma(u)} \dif u.
\end{equation}
By \eqref{e:SupF},
\begin{equation}
 F(X_t)=\int_0^t \left(\frac{b(X_s)}{\sigma(X_s)}-\frac 12 \sigma'(X_s)\right)\dif s+B_t+\alpha \sup_{0 \le s \le t} F(X_s).
\end{equation}
Denote $Y_t=F(X_t)$, it solves the following perturbed SDE:
\begin{equation}
Y_t=\int_0^t \tl b(Y_s) \dif s+ B_t+\alpha \sup_{0 \le s \le t} Y_s
\end{equation}
where $\tl b(x)=\frac{b(F^{-1}(x))}{\sigma(F^{-1}(x))}-\frac 12 \sigma'(F^{-1}(x))$.
Applying Lemma \ref{t:SthDen}, we get the following lemma about the dynamics $Y_t$:
\begin{lem} \label{t:SthDenY}
Assume that $b$ is bounded smooth and that $\sigma$ is bounded smooth with $\|\sigma'\|_\infty<\infty$, $\|\sigma''\|_\infty<\infty$  and  $\inf_{x \ge 0}|\sigma(x)|>0$. Then $\tl b$ is bounded smooth.
If $\alpha<1$, $t_0>0$ and $b$ satisfy
$$\theta(t_0,\alpha,\tl b)<1/2$$
with $\theta(r,\alpha,\tl b):=\left[\sqrt{2 \|\tl b'\|_{{\rm \infty}}^2 r^2+8 \alpha^2}+{\|\tl b'\|_{{\rm \infty}}^2 r^2+4\alpha^2}\right]$,
then
\begin{equation}
\|D Y_{t}\|_H^2 \ge  \frac{[1-2 \theta(t_0,\alpha,\tl b)] t}{2 (1+2\|\tl b'\|_\infty^2 t^2+2 \alpha^2)}, \ \ \ \ \ t \in [0,t_0].
\end{equation}
\end{lem}

\begin{proof}
It is easy to check that under the conditions in the lemma $\tl b$ is bounded smooth with $\|\tl b'\|_\infty<\infty$. Hence, the lemma immediately follows from applying Lemma \ref{t:SthDen} to $Y_t$.
\end{proof}

{\bf Acknowledgement}: We would like to gratefully thank Dr. Xiaobin Sun for helpful discussions.
\bibliographystyle{amsplain}

\end{document}